\documentclass[submission]{eptcs}

\providecommand{\event}{QPL 2025}

\input{main.sty}

\title{On Traces in Categories of Contractions\\\Large\rm Extended Abstract\\[-2ex]}

\author{Aaron David Fairbanks \qquad\qquad Peter Selinger
\institute{Dalhousie University, Canada}}

\def\authorrunning{A. D. Fairbanks and P. Selinger}
\def\titlerunning{On Traces in Categories of Contractions}

\date{Dalhousie University}

\begin{document}
\maketitle

\vspace{-0.5cm}
\begin{center}
  \emph{In memory of Phil Scott, 1947--2023}
\end{center}

\begin{abstract}
  Traced monoidal categories are used to model processes that can feed
  their outputs back to their own inputs, abstracting iteration. The
  category of finite dimensional Hilbert spaces with the direct sum
  tensor is not traced. But surprisingly, in 2014, Bartha showed that
  the monoidal subcategory of isometries is traced. The same holds for
  coisometries, unitary maps, and contractions. This suggests the
  possibility of feeding outputs of quantum processes back to their
  own inputs, analogous to iteration. In this paper, we show that
  Bartha's result is not specifically tied to Hilbert spaces, but
  works in any dagger additive category with Moore-Penrose
  pseudoinverses (a natural dagger-categorical generalization of
  inverses).
\end{abstract}

\section{Introduction}

A trace on a symmetric monoidal category $(\Cc,\oplus)$ is an
operation that assigns to each $\map{f}{A\oplus X}{B\oplus X}$ another
map $\map{\tr{X} f}{A}{B}$, satisfying some well-known axioms
{\cite{JSV1996,Malherbe-Scott-Selinger,Selinger-graphical}}. In string
diagrams, traces are represented by looping an output of $f$ back to
the corresponding input, as in the following diagram.  \vspace{-0.1cm}
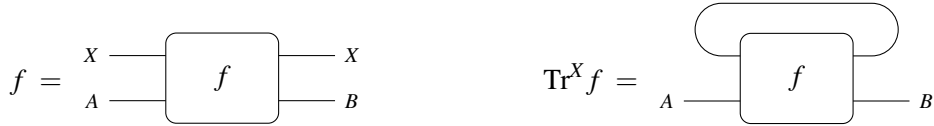
\[f \;=\;
  \begin{tikzpicture}[yscale=-1]
    \coordinate (l1) at (-1.5, -.3);
    \coordinate (l2) at (-1.5, .3);
    \coordinate (r1) at (1.5, -.3);
    \coordinate (r2) at (1.5, .3);
    \draw (l1) -- (r1);
    \draw (l2) -- (r2);
    \draw [fill=white,rounded corners] (-.75,-.6) rectangle (.75,.6);
    \node at (0,0) {$f$};
    \node [left] at (l1) {$\lab{X}$};
    \node [right] at (r1) {$\lab{X}$};
    \node [left] at (l2) {$\lab{A}$};
    \node [right] at (r2) {$\lab{B}$};
    \coordinate (n) at (0, -1);
    \coordinate (s) at (0, 1);
    \path (n) -- (s);
  \end{tikzpicture}
  \qquad\qquad\qquad
  \tr{X}f \;=\;
  \begin{tikzpicture}[yscale=-1]
    \coordinate (l2) at (-1.5, .3);
    \coordinate (r2) at (1.5, .3);
    \draw (l2) -- (r2);
    \coordinate (ssw) at (-1, -.3);
    \coordinate (sww) at (-1.35, -.65);
    \coordinate (nww) at (-1.35, -.65);
    \coordinate (nnw) at (-1, -1);
    \coordinate (nne) at (1, -1);
    \coordinate (nee) at (1.35, -.65);
    \coordinate (see) at (1.35, -.65);
    \coordinate (sse) at (1, -.3);
    \coordinate (m1) at (0, -.3);
    \draw (m1) -- (sse) to[out=0,in=90] (see) -- (nee) to[out=-90,in=0] (nne) -- (nnw) to[out=180,in=-90] (nww) -- (sww) to[out=90,in=180] (ssw) -- cycle;
    \draw [fill=white,rounded corners] (-.75,-.6) rectangle (.75,.6);
    \node at (0,0) {$f$};
    \node [left] at (l2) {$\lab{A}$};
    \node [right] at (r2) {$\lab{B}$};
    \coordinate (n) at (0, -1);
    \coordinate (s) at (0, 1);
    \path (n) -- (s);
  \end{tikzpicture} 
\vspace{-0.45cm}
\]
In categories of vector spaces, there are two relevant monoidal
structures: the ``multiplicative'' tensor $\otimes$ and the
``additive'' tensor $\oplus$, also known as biproduct or direct sum.
The multiplicative tensor on finite dimensional vector spaces has a
well-known trace (induced by the compact closed structure). But in
this paper, we are interested in the additive tensor.

A natural way to try to define an additive trace on a category of
vector spaces is by the following sum-over-paths formula, which is
motivated by the accompanying string diagrams.
\vspace{-0.1cm}
\[
\small
\qquad
\begin{array}{c@{\qquad\quad}c}
  \smallblocks
  f = \;\begin{blockarray}{ccccccccc}
    &&&&&& \lab{A} & \lab{\oplus} & \lab{X} \\[0.25ex]
    \begin{block}{ccccc(cccc)}
      \lab{B}&&& &&& f\indx{A}{B} && f\indx{X}{B} \\[0.25ex]
      \lab{\oplus}\\
      \lab{X}&&& &&& f\indx{A}{X} && f\indx{X}{X} \\
    \end{block}
  \end{blockarray}\quad
  &
  \tr{X}f = f\indx{A}{B} +
  \comp{f\indx{A}{X}}{f\indx{X}{B}} +
  \comp{f\indx{A}{X}}{f\indx{X}{X}}{f\indx{X}{B}} +
  \comp{f\indx{A}{X}}{(f\indx{X}{X})^2}{f\indx{X}{B}} + \cdots 
  \\[-1ex]
  \begin{tikzpicture}[xscale=-1,yscale=-1]
    \coordinate (l1) at (-1.65, -.6125);
    \coordinate (l2) at (-1.65, .6125);
    \node [junction] (ml1) at (-1, -.6125) {};
    \node [junction] (ml2) at (-1, .6125) {};
    \node [junction] (mr1) at (1, -.6125) {};
    \node [junction] (mr2) at (1, .6125) {};
    \coordinate (r1) at (1.65, -.6125);
    \coordinate (r2) at (1.65, .6125);
    \draw (l1) [latex'-] -- (ml1);
    \draw (l2) [latex'-] -- (ml2);
    \draw (ml1) [latex'-] -- node [arr,pos=.75] {\small $f\indx{A}{X}$} (mr2);
    \draw (ml2) [ghost] -- (mr1);
    \draw (ml1) [latex'-] -- node [arr] {\small $f\indx{X}{X}$} (mr1);
    \draw (ml2) [latex'-] -- node [arr,pos=.725] {\small $f\indx{X}{B}$} (mr1);
    \draw (ml2) [latex'-] -- node [arr] {\small $f\indx{A}{B}$} (mr2);
    \draw (mr1) [latex'-] -- (r1);
    \draw (mr2) [latex'-] -- (r2);
    \coordinate (ssw) at (-1.4, -.6125);
    \coordinate (sww) at (-1.6, -.8125);
    \coordinate (nww) at (-1.6, -1.0125);
    \coordinate (nnw) at (-1.4, -1.2125);
    \coordinate (nne) at (1.4, -1.2125);
    \coordinate (nee) at (1.6, -1.0125);
    \coordinate (see) at (1.6, -.8125);
    \coordinate (sse) at (1.4, -.6125);
    \path (mr1) -- (sse) to[out=0,in=90] (see) -- (nee) to[out=-90,in=0] (nne) -- (nnw) to[out=180,in=-90] (nww) -- (sww) to[out=90,in=180] (ssw) -- (ml1);
    \node [right] at (l1) {$\lab{X}$};
    \node [left] at (r1) {$\lab{X}$};
    \node [right] at (l2) {$\lab{B}$};
    \node [left] at (r2) {$\lab{A}$};
    \coordinate (n) at (0, -1.2125);
    \coordinate (s) at (0, 1.2125);
    \path (n) -- (s);
  \end{tikzpicture}
  &
  \begin{tikzpicture}[xscale=-1,yscale=-1]
    \coordinate (l2) at (-1.65, .6125);
    \node [junction] (ml1) at (-1, -.6125) {};
    \node [junction] (ml2) at (-1, .6125) {};
    \node [junction] (mr1) at (1, -.6125) {};
    \node [junction] (mr2) at (1, .6125) {};
    \coordinate (r2) at (1.65, .6125);
    \draw (l2) [latex'-] -- (ml2);
    \draw (ml1) [latex'-] -- node [arr,pos=.75] {\small $f\indx{A}{X}$} (mr2);
    \draw (ml2) [ghost] -- (mr1);
    \draw (ml1) [latex'-] -- node [arr] {\small $f\indx{X}{X}$} (mr1);
    \draw (ml2) [latex'-] -- node [arr,pos=.725] {\small $f\indx{X}{B}$} (mr1);
    \draw (ml2) [latex'-] -- node [arr] {\small $f\indx{A}{B}$} (mr2);
    \draw (mr2) [latex'-] -- (r2);
    \coordinate (ssw) at (-1.4, -.6125);
    \coordinate (sww) at (-1.6, -.8125);
    \coordinate (nww) at (-1.6, -1.0125);
    \coordinate (nnw) at (-1.4, -1.2125);
    \coordinate (nne) at (1.4, -1.2125);
    \coordinate (nee) at (1.6, -1.0125);
    \coordinate (see) at (1.6, -.8125);
    \coordinate (sse) at (1.4, -.6125);
    \draw (mr1) [latex'-] -- (sse) to[out=0,in=90] (see) -- (nee) to[out=-90,in=0] (nne) -- (nnw) to[out=180,in=-90] (nww) -- (sww) to[out=90,in=180] (ssw) -- (ml1);
    \node [right] at (l2) {$\lab{B}$};
    \node [left] at (r2) {$\lab{A}$};
    \coordinate (n) at (0, -1.2125);
    \coordinate (s) at (0, 1.2125);
    \path (n) -- (s);
  \end{tikzpicture}
\end{array}
\vspace{-0.4cm}
\]
The idea is similar to that of matrix multiplication, which can be
formulated as a sum over all paths from a given input coordinate to a
given output coordinate. However, the sum may not converge (supposing
there is even any notion of convergence), and so the sum-over-paths
formula does not define a total operation. Indeed, there is no totally
defined trace with respect to $\oplus$ on any category of finite (or
infinite) dimensional vector spaces {\cite{Hoshino}}.

Therefore, it came as a surprise when Bartha showed in
{\cite{Bartha_trace}} that the category of finite dimensional Hilbert
spaces and \emph{isometries} has a well-defined additive trace. In
particular, not only does Bartha's trace of an isometry always exist,
but it is again an isometry. By duality, Bartha's trace also works for
coisometries, and therefore also for unitary maps. Moreover,
Andr\'es-Mart\'inez pointed out that Bartha's trace further
generalizes to all contractions {\cite{Andres-Martinez}}. These
results suggest that there might be some physical interpretation of
loops in quantum systems, but we do not know what it is.

In this paper, we show that Bartha's result is not specifically tied
to Hilbert spaces, but works in any dagger additive category with
suitable additional structure. The specific additional structure that
we need to assume is the existence of Moore-Penrose pseudoinverses.

In a nutshell, a pseudoinverse of an arrow $\map{f}{A}{B}$ is an arrow
$\map{\mpi{f}}{B}{A}$ such that both $\comp{\mpi{f}}{f}$ and
$\comp{f}{\mpi{f}}$ are self-adjoint and $\comp{f}{\mpi{f}}{f}=f$ and
$\comp{\mpi{f}}{f}{\mpi{f}}=\mpi{f}$. Pseudoinverses are unique when
they exist, and they generalize inverses. Moreover, the definition of
pseudoinverse is purely algebraic and makes sense in any dagger category
{\cite{Cockett-Lemay}}.

Our main result is the following:

\begin{theorem}\label{thm:main}
  Given any dagger additive category with pseudoinverses, there is a totally
  defined trace on each of the following monoidal subcategories:
  \begin{itemize}
  \item the unitaries,
  \item the isometries,
  \item the coisometries, and
  \item the contractions.
  \end{itemize}
  Moreover, in the cases of unitaries and contractions, which are
  dagger monoidal subcategories, the trace is a dagger trace.
\end{theorem}

After reviewing some background material in
\cref{sec:background}, we introduce contractions in
\cref{sec:contractions} and pseudoinverses in
\cref{sec:pseudo}, and prove some of their required properties.
\cref{sec:main} is devoted to the proof of the main theorem.

\vspace{-2ex}
\paragraph{Acknowledgements.}

We thank JS Lemay and Priyaa Varshinee Srinivasan for helpful
discussions.
\vspace{-2ex}

\section{Background}\label{sec:background}

\subsection{Dagger categories}

We recall some basic definitions and properties of dagger categories
to fix the notation for the rest of the paper. For a more detailed
treatment, see {\cite{Selinger_compact,Heunen-Vicary-2019,karvonen:thesis}}.

\begin{definition}[Dagger category]
  A \emph{dagger category} is a category equipped with an
  identity-on-objects involutive contravariant functor, denoted
  $(\thg)^\dag$.  In other words, for $\map{f}{A}{B}$, we have
  $\map{f^\dag}{B}{A}$, and we have the following properties:
  \begin{itemize}
    \item $f^{\dag\dag} = f$,
    \item $(\id{A})^\dag = \id{A}$, and
    \item $(\comp{f}{g})^\dag = \comp{g^\dag}{f^\dag}$.
  \end{itemize}
\end{definition}

For example, the category $\hilbc$ of Hilbert spaces and bounded
linear maps is a dagger category. Its full subcategory $\fdhilbc$ of finite
dimensional Hilbert spaces is also a dagger category.

\begin{definition}[Properties of arrows]
  An arrow $\map{f}{A}{B}$ in a dagger category is called an
  \emph{isometry} if $\comp{f}{f^\dag} = \id{A}$, a \emph{coisometry} if
  $\comp{f^\dag}{f} = \id{B}$, and \emph{unitary} if it is an isometry
  and a coisometry. Equivalently, $f$ is unitary if it is invertible
  and $f\inv = f^\dag$. An arrow $\map{f}{A}{A}$ is \emph{self-adjoint} (or
  \emph{hermitian}) if $f=f^\dag$.
\end{definition}

In this paper, we use the symbol $\oplus$ to denote the monoidal
product, because we are mainly interested in monoidal structures that
are induced by biproducts.

\begin{definition}[Dagger monoidal category]
  A \emph{dagger monoidal category} is a dagger category that is also
  monoidal, such that $(\thg)^\dag$ is a strict monoidal functor. More
  explicitly, this means that the monoidal structure isomorphisms
  (i.e., associators and unitors) are unitary, and for all arrows $f$
  and $g$, we have
  \[
  (f \oplus g)^\dag = f^\dag \oplus g^\dag.
  \]
\end{definition}

In a dagger (monoidal) category, the isometries, the coisometries, and
the unitary maps each form a (monoidal) subcategory, i.e., they are
closed under compositions (and monoidal products).

\begin{definition}[Dagger finite biproduct category]
  A \emph{dagger finite biproduct category} is a dagger category that
  also has finite biproducts, such that the projection maps
  $\map{\pi_i}{A_1\oplus A_2}{A_i}$ and the inclusion maps
  $\map{\iota_i}{A_i}{A_1\oplus A_2}$ satisfy $\pi_i = \iota_i^\dag$.
\end{definition}

The dagger biproducts of course also form a dagger monoidal
structure. As usual in any category with finite biproducts, there is a
zero object $0$, and we can define the addition of arrows
$\map{f,g}{A}{B}$ in the usual way by
$f+g = \justmap{A}{\justxmap{A\oplus A}{\justmap{B\oplus
      B}{B}}{f\oplus g}}$.  There are also zero maps
$\map{0}{A}{\justmap{0}{B}}$. Indeed, every (dagger) finite biproduct
category is enriched over commutative monoids.

\begin{definition}[Negatives, additive category]
  We say that a (dagger) finite biproduct category \emph{has
  negatives} if for every $\map{f}{A}{B}$, there exists
  $\map{-f}{A}{B}$ such that $f + (-f) = 0$. A (dagger) finite
  biproduct category with negatives is called a (dagger)
  \emph{additive} category.
\end{definition}

\begin{definition}[Positive map]
  A map $\map{a}{A}{A}$ in a dagger category is \emph{positive} if
  there exists $\map{f}{A}{B}$ with $a = \comp{f}{f^\dag}$.
\end{definition}

Positive maps in $\hilbc$ are positive operators in the usual sense.
Every positive map is self-adjoint, and the sum of positive maps is
positive if we have dagger biproducts. Given two maps
$\map{f,g}{A}{A}$, in a dagger finite biproduct category, we say that
$f\leq g$ if there exists some positive $a$ such that $g=f+a$.

\begin{lemma}\label{lem:conjugate}
  Let $\map{f,g}{A}{A}$ and assume $f\leq g$. (a) For all
  $\map{h}{A}{B}$, we have
  $\comp{h^\dag}{f}{h}\leq \comp{h^\dag}{g}{h}$. (b) For all
  $\map{f',g'}{A'}{A'}$ with $f'\leq g'$, we have
  $f\oplus f'\leq g\oplus g'$.
\end{lemma}

\begin{proof}
  Since $g=f+a$ and $g' = f'+a'$ for some positive $a$ and $a'$, we have
  \[\comp{h^\dag}{g}{h}=\comp{h^\dag}{f}{h}+\comp{h^\dag}{a}{h} \quad
    \text{and} \quad g \oplus g' = (f + a) \oplus (f'+ a') = (f \oplus f') + (a \oplus
  a').\] It is easy to see $\comp{h^\dag}{a}{h}$ and $a \oplus a'$
  are positive, which implies both claims.
\end{proof}

\subsection{Matrices}

In a finite biproduct category, consider objects $A = A_1
\oplus \cdots \oplus A_m$ and $B=B_1 \oplus \cdots \oplus B_n$. It is
well-known that maps $\map{f}{A}{B}$ are in one-to-one correspondence with
matrices $(f\indx{i}{j})$, where $\map{f\indx{i}{j}}{A_i}{B_j}$. We write
\[
\smallblocks
f\; = \;\;\; \begin{blockarray}{ccccccccccc}
  &&&&&& \lab{A_1} & \lab{\oplus} & \lab{\cdots} & \lab{\oplus} & \lab{A_m} \\[0.25ex]
  \begin{block}{ccccc(cccccc)}
    \lab{B_1}&&& &&& f\indx{1}{1} && \cdots && f\indx{m}{1} \\
    \lab{\oplus}\\
    \raisebox{.5ex}{\labvdots}&&& &&& \raisebox{.25ex}{$\vdots$} && \,\;\raisebox{.25ex}{$\ddots$}\;\, && \raisebox{.25ex}{$\vdots$} \\[0.25ex]
    \lab{\oplus}\\
    \lab{B_n}&&& &&& f\indx{1}{n} && \cdots && f\indx{m}{n} \\
  \end{block}
\end{blockarray}~~~.
\vspace{-0.25cm}
\]
Then composition of morphisms coincides with the usual formula for
matrix multiplication, given by
\[
(\comp{f}{g})\indx{i}{k} = \sum_{j \in J}\comp{f\indx{i}{j}}{g\indx{j}{k}}.
\]
In a dagger finite biproduct category, the dagger of a morphism
coincides with the adjoint of its matrix:
\[
(f^\dag)\indx{j}{i} = (f\indx{i}{j})^\dag.
\]

Given a matrix
\[ \map{f = \pmat{f\indx{1}{1}&f\indx{2}{1}\\f\indx{1}{2}&f\indx{2}{2}}}{A_1\oplus A_2}{B_1\oplus B_2},
\]
we call each $\map{f\indx{i}{j}}{A_i}{B_j}$ a \emph{component} of $f$, we call
$\map{\psmall{f\indx{i}{1}\\f\indx{i}{2}}}{A_i}{B_1\oplus B_2}$ a \emph{column} of $f$, and
we call $\map{\pmat{f\indx{1}{j}~f\indx{2}{j}}}{A_1\oplus A_2}{B_j}$ a \emph{row}
of $f$. We use analogous terminology for larger matrices.

\begin{remark}
  In a dagger finite biproduct category, an arrow of the form
  \[
  \smallblocks
  \begin{blockarray}{ccccccccccc}
    &&&&&& \lab{A_1} & \lab{\oplus} & \lab{\cdots} & \lab{\oplus} & \lab{A_m} \\[0.25ex]
    \begin{block}{ccccc(cccccc)}
      \lab{B}&&& &&& f_1 && \cdots && f_m \\
    \end{block}
  \end{blockarray}~~~
  \vspace{-0.25cm}
  \]
  is an isometry if and only if $\comp{f_i}{f_i^\dag} = \id{A_i}$ for
  all $i$ and $\comp{f_i}{f_j^\dag} = 0$ for $i\neq j$.
\end{remark}

\begin{remark}\label{rem:isometry_component}
  In a dagger additive category, every isometry
  is a component of a unitary. Indeed, suppose $\map{f}{A}{B}$ is an
  isometry. Then the following arrow is unitary.
  \[
  \smallblocks
  \begin{blockarray}{ccccccccc}
      &&&&&& ~~~~\lab{A} & \hspace{.5cm}\lab{\oplus} & \lab{B} \\[0.25ex]
      \begin{block}{ccccc(cccc)}
        \lab{A}&&& &&& ~~~~0 && f^\dag \\[0.25ex]
        \lab{\oplus}\\
        \lab{B}&&& &&& ~~~~f && \id{B} - \comp{f^\dag}{f} \\
      \end{block}
    \end{blockarray}
  \vspace{-0.25cm}
  \]
\end{remark}

\subsection{Dagger idempotents}

\begin{definition}[Dagger idempotents]
  A arrow $\map{p}{A}{A}$ is called a \emph{dagger idempotent} (or
  \emph{projection}) if $p = \comp{p}{p}=p^{\dag}$.
\end{definition}

Whenever $\map{f}{B}{A}$ is an isometry, then $p = \comp{f^\dag}{f}$
is a dagger idempotent. If $p$ is of this form, we say that $p$ is
\emph{dagger split}. When dagger splittings exist, they are unique up
to unitary isomorphism. It is well-known that every dagger category
can be fully embedded in a dagger category with all dagger splittings,
called its \emph{dagger idempotent completion}
{\cite{Selinger-idem}}. Moreover, all structure of interest (e.g.,
monoidal structure, biproducts, addition, negatives, and, as we will
later introduce, pseudoinverses \cite{Cockett-Lemay}) on a dagger
category transports to its dagger idempotent completion.

\begin{definition}[Complementary idempotents]
  Two idempotents $\map{p,q}{A}{A}$ in a finite biproduct category are
  \emph{complementary} if $p + q = \id{A}$ and $\comp{p}{q} = 0 =
  \comp{q}{p}$.
\end{definition}

If the category has negatives, complements always exist, because
whenever $p$ is a (dagger) idempotent, so is $1-p$. It is obvious that
the complement is unique in that case. Interestingly, uniqueness even
holds without assuming negatives, because if both $q_1$ and $q_2$ are
complements of $p$, we have $q_1 = \comp{q_1}{(p + q_2)} =
\comp{q_1}{q_2} = \comp{(p + q_1)}{q_2} = q_2$.

Complementary dagger idempotents are an algebraic abstraction
of orthogonal complement subspace projections. 

\begin{lemma}[Direct sum decomposition]\label{lem:proj_decompose}
  Consider a (dagger) finite biproduct category in which all (dagger)
  idempotents (dagger) split. Given an object $A$ with complementary
  idempotents $\map{p,q}{A}{A}$, there exist objects $A_1,A_2$ with
  $A = A_1 \oplus A_2$ such that
  \[
  \smallblocks
  p = \;\begin{blockarray}{ccccccccc}
    &&&&&& \lab{A_1} & \lab{\oplus} & \lab{A_2}\hpad \\[0.25ex]
    \begin{block}{ccccc(cccc)}
      \lab{A_1}&&& &&& \justid && 0\vpad\hpad \\[0.25ex]
      \lab{\oplus}\\
      \lab{A_2}&&& &&& 0 && 0\hpad \\
    \end{block}
  \end{blockarray}\;\;\;
  \quad\mbox{and}\quad
  q = \;\begin{blockarray}{ccccccccc}
    &&&&&& \lab{A_1} & \lab{\oplus} & \lab{A_2}\hpad \\[0.25ex]
    \begin{block}{ccccc(cccc)}
      \lab{A_1}&&& &&& 0 && 0\vpad\hpad \\[0.25ex]
      \lab{\oplus}\\
      \lab{A_2}&&& &&& 0 && \justid\hpad \\
    \end{block}
  \end{blockarray}\;\;\;.
  \vspace{-3ex}
  \]
  Moreover, the factorization is unique up to (unitary)
  isomorphisms of the direct sum factors.
\end{lemma}

\begin{proof}[Proof idea]
  Let $A_1$ be a splitting of $p$, and let $A_2$ be a splitting of
  $q$. The claimed properties are easy to verify.
\end{proof}

\begin{remark}
  In \cref{lem:proj_decompose} and elsewhere, we write $A=A_1\oplus
  A_2$ instead of $A\iso A_1\oplus A_2$; this is justified because
  (dagger) biproducts are defined up to (unitary) isomorphism in the
  first place.
\end{remark}

\subsection{Trace}

A \emph{trace} on a symmetric monoidal category $\Cc$ is a family of
operations $\map{\tr{X}}{\Cc(A\oplus X,B\oplus X)}{\Cc(A,B)}$, subject
to a small number of axioms that can be found in
{\cite{JSV1996,Malherbe-Scott-Selinger,Selinger-graphical}}. The
concept of a \emph{partial trace} is defined similarly, except that
$\tr{X}$ is a partially defined operation
{\cite{Haghverdi-Scott}}. The axioms are such that a partially traced
category in which the trace happens to be totally defined is a
(totally) traced category. It was shown in
{\cite{Malherbe-thesis,Malherbe-Scott-Selinger}} that every partially
traced category can be faithfully embedded in a totally traced one,
and conversely, every monoidal subcategory of a totally traced
category is partially traced.

We will make use of the following construction, which can be found in
{\cite{Malherbe-Scott-Selinger}}. It is remarkable because it works in
any additive category.

\begin{definition}[Kernel-image trace]\label{def:kernel-image}
  Let $\map{f}{A \oplus X}{B \oplus X}$ be an arrow in an additive
  category. The \emph{kernel-image trace} $\map{\kitr{X} f}{A}{B}$ is
  defined if there exist arrows $\map{i}{A}{X}$ and $\map{k}{X}{B}$
  such that
  \[
    f\indx{A}{X} = \comp{i}{(\id{X} - f\indx{X}{X})}
    \qquad\text{and}\qquad
    \comp{(\id{X} - f\indx{X}{X})}{k} = f\indx{X}{B},
  \]
  as in the following commutative diagram
  \vspace{-2ex}
  \[
    \begin{tikzpicture}[yscale=0.9]
      \node (a) at (-1.5,1) {$A$};
      \node (b) at (.5,1) {$X$};
      \node (c) at (-.5,-1) {$X$};
      \node (d) at (1.5,-1) {$B$.};
      \draw[->] (b) -- node[fill=white]{$\id{X} - f\indx{X}{X}$} (c);
      \draw[->] (a) -- node[below left]{$f\indx{A}{X}$} (c);
      \draw[->] (b) -- node[above right]{$f\indx{X}{B}$} (d);
      \draw[->, dashed] (a) -- node[above]{$i$} (b);
      \draw[->, dashed] (c) -- node[below]{$k$} (d);
    \end{tikzpicture}
    \vspace{-2ex}
  \]
  In this case, we define
  \[\kitr{X} f = f\indx{A}{B} + \comp{i}{(\id{X} - f\indx{X}{X})}{k}.\]
  (Otherwise, the kernel-image trace is undefined.) Note $\kitr{X}$ is
  independent of the choice of each $i$ and $k$, since
  \[f\indx{A}{B} + \comp{i}{f\indx{X}{B}} = \kitr{X} f = f\indx{A}{B} + \comp{f\indx{A}{X}}{k}.\]
\end{definition}

\begin{proposition}[{\cite{Malherbe-Scott-Selinger}}]
The kernel-image trace is a partial trace.
\end{proposition}

\begin{remark}\label{rem:dagger-trace}
  In a dagger category, a (partial) trace is called a \emph{dagger
    (partial) trace} if $\justtr(f^{\dag})=(\justtr f)^\dag$. In a
  dagger additive category, the kernel-image trace is always a dagger
  partial trace, because its definition is self-dual.
\end{remark}

\section{Contractions}\label{sec:contractions}

\subsection{Basic properties}

In the category of Hilbert spaces, a \emph{contraction} is a map
$\map{f}{A}{B}$ such that for all $v\in A$,
$\norm{\app{f}{v}}\leq\norm{v}$. The following definition generalizes
this concept to arbitrary dagger additive categories.

\begin{definition}[Contraction]\label{def:contraction}
  A \emph{contraction} in a dagger additive category is an arrow
  $\map{f}{A}{B}$ such that $\comp{f}{f^\dag} \leq \id{A}$.  In other
  words, such that there exists an arrow $\map{g}{A}{B'}$ with
  $\comp{f}{f^\dag} + \comp{g}{g^\dag} = \id{A}$. Note that this is
  the case if and only if the map $\map{\psmall{f\\g}}{A}{B\oplus B'}$
  is an isometry. A \emph{cocontraction} is defined dually.
\end{definition}

In particular, every isometry, coisometry, and unitary map is a
contraction. Also, biproduct projections and injections are
contractions.

Note that \cref{def:contraction} could be stated even without assuming
negatives, but most of the useful properties of contractions rely on
additivity. A point in case is the next proposition, which gives
several alternative characterizations of contractions, none of which
would be equivalent in the absence of negatives (see
counterexamples~{\ref{cex:non-cocontraction}} and
{\ref{cex:non-iso-then-coiso}}).

\begin{proposition}[Characterizations of contractions]\label{prop:contractions}
  Let $\map{f}{A}{B}$ be an arrow in a dagger additive category. The
  following are equivalent.
  \begin{enumerate}[(a)]
  \item\label{item:c1} $f$ is a component of a unitary.
  \item\label{item:c2} $f$ is a contraction.
  \item\label{item:c3} $f$ is a cocontraction.
  \item\label{item:c4} $f$ is of the form $\comp{m}{e}$ for some isometry $\map{m}{A}{X}$
    and coisometry $\map{e}{X}{B}$.
  \item\label{item:c5} $f$ is a composition of isometries and coisometries.
  \end{enumerate}
\end{proposition}

We delay the proof until we have established some lemmas. The
following lemma tells us that contractions, like isometries, form a
monoidal subcategory.

\begin{lemma}\label{lem:contraction-composition}
  Contractions are closed under composition and monoidal products.
\end{lemma}

\begin{proof}
  For composition, let $\map{f}{A}{B}$ and $\map{g}{B}{C}$ be contractions. Then
  $\comp{f}{f^\dag}\leq \id{A}$ and $\comp{g}{g^\dag}\leq \id{B}$. Using
  \cref{lem:conjugate}(a), we get
  \[
  \comp{(\comp{f}{g})}{(\comp{f}{g})^\dag} = \comp{f}{g}{g^\dag}{f^\dag} \leq \comp{f}{\id{B}}{f^\dag} = \comp{f}{f^\dag} \leq \id{A}.
  \]
  Therefore, $\comp{g}{f}$ is a contraction. For monoidal products, let
  $\map{f}{A}{B}$ and $\map{g}{A'}{B'}$ be contractions. Using
  \cref{lem:conjugate}(b), we get
  \[
  \comp{(f\oplus g)}{(f\oplus g)^\dag} = (\comp{f}{f^\dag})\oplus(\comp{g}{g^\dag}) \leq \id{A}\oplus \id{A'} = \id{A \oplus A'}.
  \]
  Therefore, $f\oplus g$ is a contraction.
\end{proof}

\begin{lemma}[Contractions as components of unitaries]\label{lem:selfdual}
  In a dagger additive category, contractions
  are precisely the components of unitaries. In particular,
  contractions coincide with cocontractions.
\end{lemma}

\begin{proof}
  First, a component of a unitary is a composition of three
  contractions $u_{jk} = \comp{\iota_k}{u}{\pi_j}$, and is therefore a
  contraction itself. Conversely, every contraction is a component of
  an isometry (as remarked in \cref{def:contraction}), which in turn
  is a component of a unitary by
  \cref{rem:isometry_component}. Finally, since being a component of a
  unitary is a self-dual concept, so is being a contraction.
\end{proof}

We can now prove \cref{prop:contractions}.

\begin{proof}[Proof of \cref{prop:contractions}]
  The equivalence $\ref{item:c1} \iff \ref{item:c2} \iff \ref{item:c3}$ is \cref{lem:selfdual}. For
  $\ref{item:c2} \implies \ref{item:c4}$, assume $\comp{f}{f^\dag}+\comp{g}{g^\dag}=\justid$.
  Then $f=\comp{m}{e}$, where $e=\pmat{\justid~~0\,}$ is a coisometry and
  $m=\psmall{f\\g}$ is an isometry. The implication $\ref{item:c4}\implies\ref{item:c5}$
  is trivial, and $\ref{item:c5}\implies \ref{item:c2}$ follows because contractions are
  closed under composition by \cref{lem:contraction-composition}.
\end{proof}

\subsection{Contractions and definiteness}

Contractions have even better properties when the underlying dagger
category satisfies the following condition.

\begin{definition}[Definite]
  A dagger category with a zero object is \emph{definite} if for all
  arrows $f$, we have that $\comp{f}{f^\dag} = 0$ implies $f = 0$.
\end{definition}

In the familiar context of Hilbert spaces, the columns or rows of a
contraction have norm at most $1$. An analogue of this principle holds
in any definite dagger additive category.

\begin{lemma}[Maxed-out column]\label{lem:one_gives_zero}
  In a definite dagger additive category, assume $f=\psmall{f_1\\f_2}$
  is a contraction. If $f_1$ is an isometry, then $f_2=0$.
\end{lemma}

\begin{proof}
  It suffices to show the result when $f$ is an isometry, because
  every contraction is a row of an isometry. We have $\id{A} =
  \comp{f}{f^\dag} = \comp{f_1}{f_1^\dag} + \comp{f_2}{f_2^\dag} =
  \id{A} + \comp{f_2}{f_2^\dag}$. Subtracting $\id{A}$ from both
  sides, we get $0 = \comp{f_2}{f_2^\dag}$. Now by definiteness, $f_2
  = 0$.
\end{proof}

\begin{corollary}[Maxed-out row and column]\label{cor:maxed}
  In a definite dagger additive category,
  assume
  \[ f = \pmat{\justid & f_{12} \\ f_{21} & f_{22}}
  \]
  is a contraction. Then $f_{12}=0$ and $f_{21}=0$.
\end{corollary}

\begin{proof}
  $f_{12}=0$ follows from \cref{lem:one_gives_zero} and $f_{21}=0$
  follows from its dual.
\end{proof}

The first part of the following is basically \cref{cor:maxed} in more
algebraic language. The second part amounts to the observation that
the fixed points of a contraction $f$ are also fixed by $f^\dag$.

\begin{corollary}[Fixed points of contraction]\label{lem:contract_project}
  Suppose $\map{f}{A}{A}$ is a contraction and $\map{p}{A}{A}$ is a
  dagger idempotent in a definite dagger additive category.
  \begin{enumerate}[(a)]
  \item\label{item:fpa} If $\comp{p}{f}{p} = p$, then $\comp{p}{f} = p = \comp{f}{p}$.
  \item\label{item:fpb}$\comp{p}{f} = p$ if and only if $\comp{f}{p} = p$.
  \end{enumerate}
\end{corollary}
\begin{proof}
  Without loss of generality, we can assume all dagger idempotents
  split, because otherwise we can pass to the dagger idempotent
  completion.  Let $A=A_1\oplus A_2$ be the decomposition of $A$
  obtained by splitting $p$ and its complement as in
  \cref{lem:proj_decompose}.  Write
  \vspace{-1ex}
  \[\smallblocks
  f \; = \;\;\; \begin{blockarray}{ccccccccc}
    &&&&&& \lab{A_1} & \lab{\oplus} & \lab{A_2} \\[0.25ex]
    \begin{block}{ccccc(cccc)}
      \lab{A_1}&&& &&& f_{11} && f_{12} \\[0.25ex]
      \mathclap{\lab{\oplus}}\\
      \lab{A_2}&&& &&& f_{21} && f_{22} \\
    \end{block}
  \end{blockarray}\;\;\;.
  \vspace{-0.25cm}
  \vspace{-1ex}
  \]
  To prove \ref{item:fpa}, note that $\comp{p}{f}{p}=p$ means that
  $f_{11}=1$, which by \ref{cor:maxed} implies that $f_{12}=0$ and
  $f_{21}=0$, hence $\comp{p}{f} = p = \comp{f}{p}$. Claim
  \ref{item:fpb} follows from \ref{item:fpa}.
\end{proof}

\section{Pseudoinverses}\label{sec:pseudo}

\subsection{Definition of pseudoinverse}

Every linear map $\map{f}{V}{W}$ between finite dimensional Hilbert
spaces is of the form
\vspace{-1ex}
\[
\smallblocks
f\; = \;\;\;
\begin{blockarray}{ccccccccc}
  &&&&&& \lab{(\ker f)^\perp} & \lab{\oplus} & \lab{\phantom{(}\ker f\phantom{)^\perp}} \\[0.25ex]
  \begin{block}{ccccc(cccc)}
    \lab{\im f}&&& &&& a && 0 \\[0.25ex]
    \lab{\oplus}\\
    \lab{(\im f)^\perp}&&& &&& 0 && 0 \\
  \end{block}
  \end{blockarray}\;\;\;,
\vspace{-0.25cm}
\vspace{-1ex}
\]
where $a$ is invertible. This section is about dagger additive
categories in which an analogous fact holds.
Observe that, given the above decomposition of $\map{f}{V}{W}$, we
automatically get a map $\map{\mpi{f}}{W}{V}$ in the other
direction via
\[
\smallblocks
\mpi{f}\; = \;\;\;
\begin{blockarray}{ccccccccc}
  &&&&&& \lab{\phantom{(^\perp}\im f\phantom{)^\perp}} & \lab{\oplus} & \lab{\phantom{{}^\perp}(\im f)^\perp} \\[0.25ex]
  \begin{block}{ccccc(cccc)}
    \lab{(\ker f)^\perp}&&& &&& a^{-1} && 0 \\[0.25ex]
    \lab{\oplus}\\
    \lab{\ker f}&&& &&& 0 && 0 \\
    \end{block}
\end{blockarray}\;\;\,.  \vspace{-0.25cm}
\]
We note that this ``almost inverse'' $\mpi{f}$ of $f$ satisfies the
following four properties:
\begin{equation}\label{eqn:mpi}
  f = \comp{f}{\mpi{f}}{f},\qquad
  \mpi{f} = \comp{\mpi{f}}{f}{\mpi{f}}, \qquad
  \comp{f}{\mpi{f}} = (\comp{f}{\mpi{f}})^\dag,\qquad
  \comp{\mpi{f}}{f} = (\comp{\mpi{f}}{f})^\dag.
\end{equation}
It so happens that these four laws uniquely determine $\mpi{f}$ given
$f$.

\begin{definition}[Pseudoinverse]
  In a dagger category, a \emph{pseudoinverse} (or \emph{Moore-Penrose
  pseudoinverse}) of a map $\map{f}{A}{B}$ is an arrow
  $\map{\mpi{f}}{B}{A}$ such that the equations \eqref{eqn:mpi} hold.
  A \emph{pseudoinverse dagger category} (in \cite{Cockett-Lemay},
  \emph{Moore-Penrose dagger category}) is a dagger category in which
  every arrow has a pseudoinverse.
\end{definition}

Before we prove uniqueness, here is a bit of background on
pseudoinverses. They were introduced by Moore in \cite{Moore} and
rediscovered by Penrose in \cite{Penrose}. For an overview, see
\cite{Ben-Israel} or \cite{Baksalary-Trenkler}. Pseudoinverses were
studied in abstract dagger categories by Puystjens and Robinson in
\cite{Puystjens-Robinson_factorization, Puystjens-Robinson_additive,
  Puystjens-Robinson_ep, Puystjens-Robinson_kernels,
  Puystjens-Robinson_symmetric} and recently by Cockett and Lemay in
\cite{Cockett-Lemay}.

\begin{example}\label{exa:hilb-pseudo}
  In $\hilbc$, an arrow $\map{f}{\hsp{H}}{\hsp{H}'}$ is
  pseudoinvertible if and only if the image of $f$ is closed. In
  $\fdhilbc$, every arrow is pseudoinvertible.
\end{example}

We note the following equivalent characterization of pseudoinverses;
it will simplify the proof of uniqueness in \cref{prop:mp_unique}
below.

\begin{lemma}[Second definition of pseudoinverse]\label{lem:mpi-alt}
  Pseudoinverses $f$ and $\mpi{f}$ in a dagger category are equivalently
  characterized by the equations
  \begin{equation}\label{eqn:mpi-alt}
  f = \comp{f}{f^\dag}{\mpi{f}^\dag},\qquad
  f = \comp{\mpi{f}^\dag}{f^\dag}{f},\qquad
  \mpi{f} = \comp{\mpi{f}}{\mpi{f}^\dag}{f^\dag},\qquad
  \mpi{f} = \comp{f^\dag}{\mpi{f}^\dag}{\mpi{f}}.
  \end{equation}
\end{lemma}

\begin{proof}
  From {\eqref{eqn:mpi-alt}}, we derive
  \[
  \comp{\mpi{f}}{f} =\comp{\mpi{f}}{f}{f^\dag}{\mpi{f}^\dag} =
  \comp{f^\dag}{\mpi{f}^\dag}\quad\text{and}\quad \comp{f}{\mpi{f}} =
  \comp{f}{\mpi{f}}{\mpi{f}^\dag}{f^\dag} =
  \comp{\mpi{f}^\dag}{f^\dag},
  \]
  i.e.,
  $\comp{\mpi{f}}{f} = (\comp{\mpi{f}}{f})^\dag$ and
  $\comp{f}{\mpi{f}} = (\comp{f}{\mpi{f}})^\dag$.
  Hence the two definitions are equivalent as $f$ and $\mpi{f}$ are
  permitted to slide past each other, picking up daggers.
\end{proof}

\begin{proposition}[Uniqueness of pseudoinverse]\label{prop:mp_unique}
  If $\mpi{f}$ and $\altmpi{f}$ are both pseudoinverses of
  $f$, then $\mpi{f} = \altmpi{f}$.
\end{proposition}
\begin{proof}
  $\mpi{f} = \comp{\mpi{f}}{\mpi{f}^\dag}{f^\dag} =
    \comp{\mpi{f}}{\mpi{f}^\dag}{f^\dag}{f}{\altmpi{f}} = \comp{\mpi{f}}{f}{\altmpi{f}}$.
  Symmetrically, $\altmpi{f} = \comp{\mpi{f}}{f}{\altmpi{f}}$.
\end{proof}

Note that the notion of pseudoinverse is self-dual and therefore
respected by dagger: if $\map{f}{A}{B}$ is pseudoinvertible, then so
is $f^\dag$ with $\mpi{(f^\dag)} = (\mpi{f})^\dag$. Also note that if
$f$ is pseudoinvertible, then $\comp{\mpi{f}}{f}$ and
$\comp{f}{\mpi{f}}$ are dagger idempotents. More specifically,
$\comp{\mpi{f}}{f}$ represents projection onto the image of $f$, and
$\comp{f}{\mpi{f}}$ represents projection onto the coimage of $f$
(i.e., the orthogonal complement of the kernel). We hence obtain the
following decomposition, which is analogous to what happens in $\fdhilbc$.

\begin{proposition}[Generalized singular value decomposition {\cite{Cockett-Lemay}}]\label{prop:svd}
  Let $\map{f}{A}{B}$ be an arrow in a dagger additive category in
  which all dagger idempotents split. Then $f$ is pseudoinvertible if
  and only if we can write $A= A_1\oplus A_2$ and $B=
  B_1\oplus B_2$ such that
  \vspace{-0.5ex}
  \[\smallblocks
  f \; = \;\;\; \begin{blockarray}{ccccccccc}
    &&&&&& \lab{A_1} & \lab{\oplus} & \lab{A_2}\hpad \\[0.25ex]
    \begin{block}{ccccc(cccc)}
      \lab{B_1}&&& &&& a && 0\vpad\hpad \\[0.25ex]
      \mathclap{\lab{\oplus}}\\
      \lab{B_2}&&& &&& 0 && 0\hpad \\
    \end{block}
  \end{blockarray}\;\;\;
  \qquad\mbox{and}\qquad
  \mpi{f} \; = \;\;\; \begin{blockarray}{ccccccccc}
    &&&&&& \lab{B_1} & \lab{\oplus} & \lab{B_2}\hpad \\[0.25ex]
    \begin{block}{ccccc(cccc)}
      \lab{A_1}&&& &&& a\inv && 0\vpad\hpad \\[0.25ex]
      \mathclap{\lab{\oplus}}\\
      \lab{A_2}&&& &&& 0 && 0\hpad \\
    \end{block}
  \end{blockarray}\;\;\;,
  \vspace{-0.45cm}
  \]
  where $\map{a}{A_1}{B_1}$ is invertible.
  Moreover, the factorization of $f$ is unique up to unitary
  isomorphisms of the direct sum factors.
\end{proposition}

\begin{proof}
  Clearly, if $f$ can be written in the stated form, then $f$ is
  pseudoinvertible with pseudoinverse as stated. For the left-to-right
  implication, assume $f$ is pseudoinvertible. Consider the dagger
  idempotents $\map{\comp{f}{\mpi{f}}}{A}{A}$ and
  $\map{\comp{\mpi{f}}{f}}{B}{B}$. By splitting them and their
  complements as in \cref{lem:proj_decompose}, we can write
  $A=A_1\oplus A_2$ and $B=B_1\oplus B_2$, where $\comp{f}{\mpi{f}} =
  \comp{\pi_1^A}{\iota_1^A}$ and $\comp{\mpi{f}}{f} =
  \comp{\pi_1^B}{\iota_1^B}$. Let
  $a=\map{\comp{\iota_1^A}{f}{\pi_1^B}}{A_1}{B_1}$. Then $f =
  \comp{f}{\mpi{f}}{f}{\mpi{f}}{f} =
  \comp{\pi_1^A}{\iota_1^A}{f}{\pi_1^B}{\iota_1^B} =
  \comp{\pi_1^A}{a}{\iota_1^B}$, hence $f$ is of the claimed
  form. Moreover, it is easy to verify that $a\inv =
  \comp{\iota_1^B}{\mpi{f}}{\pi_1^A}$. Uniqueness is as in
  Lemma~\ref{lem:proj_decompose}.
\end{proof}

\subsection{EP maps}

The generalized singular value decomposition of \cref{prop:svd}
is especially nice if $f$ is a so-called EP-map, which we now
define. This definition captures the notion of an endomorphism whose
kernel and image are orthogonal complements.

\begin{definition}[EP maps]
  An \emph{EP map} (or \emph{range hermitian map}) in a dagger
  category is a pseudoinvertible endomorphism $\map{f}{A}{A}$ such
  that $\comp{f}{\mpi{f}} = \comp{\mpi{f}}{f}$.
\end{definition}

The term ``EP'' was introduced by
Schwerdtfeger~\cite{Schwerdtfeger1950}, who does not explain what
these letters stand for. Given that $\comp{f}{\mpi{f}}$ and
$\comp{\mpi{f}}{f}$ are projections that are equal to each other, a
useful mnemonic is that EP stands for ``equal projections''.

\begin{remark}[Normal operators are EP]
  If $f$ is pseudoinvertible and $\comp{f}{f^\dag} =
  \comp{f^\dag}{f}$, then $f$ is EP:
  \[
  \comp{f}{\mpi{f}}
  = \comp{\mpi{f}^\dag}{f^\dag}\\
  = \comp{\mpi{f}^\dag}{\mpi{f}}{f}{f^\dag}\\
  = \comp{\mpi{(\comp{f}{f^\dag})}}{f}{f^\dag}\\
  = \comp{\mpi{(\comp{f^\dag}{f})}}{f^\dag}{f}\\
  = \comp{\mpi{f}}{\mpi{f}^\dag}{f^\dag}{f}\\
  = \comp{\mpi{f}}{f}.
  \]
\end{remark}

The following proposition characterizes EP maps in the style of
\cref{prop:svd}.

\begin{proposition}\label{prop:ep-svd}
  Let $\map{f}{A}{A}$ be an arrow in a dagger additive category in which
  all dagger idempotents split. Then $f$ is EP if and only if we can
  write $A= A_1\oplus A_2$ such that
  \vspace{-1ex}
  \[\smallblocks
  f \; = \;\;\; \begin{blockarray}{ccccccccc}
    &&&&&& \lab{A_1} & \lab{\oplus} & \lab{A_2}\hpad \\[0.25ex]
    \begin{block}{ccccc(cccc)}
      \lab{A_1}&&& &&& a && 0\vpad\hpad \\[0.25ex]
      \mathclap{\lab{\oplus}}\\
      \lab{A_2}&&& &&& 0 && 0\hpad \\
    \end{block}
  \end{blockarray}\;\;\;,
  \vspace{-0.25cm}
  \vspace{-2ex}
  \]
  where $\map{a}{A_1}{A_1}$ is invertible.
\end{proposition}

\begin{proof}
  Like the proof of \cref{prop:svd}, but using the
  fact that the idempotents $\comp{\mpi{f}}{f}$ and $\comp{f}{\mpi{f}}$ are
  equal and therefore have the same splitting.
\end{proof}

Before we say more about EP maps, we need the following lemma.

\begin{lemma}\label{lem:zero}
  In a dagger category with a zero object, if $\map{f}{A}{B}$ is
  pseudoinvertible and $\comp{f}{f^\dag} = 0$, then $f = 0$. In
  particular, every pseudoinverse dagger category with a zero object
  is definite.
\end{lemma}

\begin{proof}
  Using {\eqref{eqn:mpi-alt}} from \cref{lem:mpi-alt}, we have $f =
  \comp{f}{f^\dag}{\mpi{f}^\dag} = 0$.
\end{proof}

We saw in \cref{lem:contract_project} that the fixed points of a
contraction $f$ are also fixed by $f^\dag$. The following lemma is the
same fact in different language: $g=1-f$ being EP means that
$1-\comp{g}{\mpi{g}}$ (the projection onto the fixed points of $f$) is
equal to $1-\comp{\mpi{g}}{g}$ (the projection onto the fixed points of
$f^\dag$).

\begin{lemma}[Contractions and EP
  maps]\label{lem:contraction_complementary}
  Let $\map{f}{A}{A}$ be a contraction in a pseudoinverse dagger
  additive category. Then $g = \id{A}-f$ is EP.
\end{lemma}

\begin{proof}
  Observe that
  $\comp{(\id{A}-\comp{g}{\mpi{g}})}{(\id{A}-g)} =
  \id{A}-\comp{g}{\mpi{g}}$. By \cref{lem:zero}, the category is
  definite, and thus we can apply \cref{lem:contract_project} to obtain
  $\comp{(\id{A}-g)}{(\id{A}-\comp{g}{\mpi{g}})} =
  \id{A}-\comp{g}{\mpi{g}}$. Simplifying, we get
  $\comp{g}{g}{\mpi{g}}=g$. Similarly, $\comp{\mpi{g}}{g}{g}=g$, hence
  $\comp{g}{\mpi{g}} =
  \comp{\mpi{g}}{g}{g}{\mpi{g}}=\comp{\mpi{g}}{g}$, as claimed.
\end{proof}

\section{Proof of the main result}
\label{sec:main}

The purpose of this section is to prove \cref{thm:main}. That is, in a
pseudoinverse dagger additive category, the monoidal subcategories of
unitaries, isometries, coisometries, and contractions are traced. The
proof in the case of isometries proceeds in two steps: In
\cref{lem:step1}, we show that the kernel-image trace of a contraction
(and therefore, of an isometry) is always defined. This is the only
part of the proof that uses pseudoinverses. In \cref{lem:step2}, we
show that the kernel-image trace of an isometry is again an
isometry. These two facts imply that the kernel-image trace is totally
defined on the category of isometries. Since it is already known to be
a partial trace, these facts are sufficient to prove that the category
of isometries is totally traced. The case of contractions is proved
similarly, and the other cases are easy consequences.

\begin{lemma}[Trace is defined for contractions]
  \label{lem:kerim_total}
  \label{lem:step1}
  In a pseudoinverse dagger additive category,
  the kernel-image trace is always defined for contractions.
\end{lemma}

\begin{proof}
  Let $\map{f}{A \oplus X}{B \oplus X}$ be a contraction. Then
  $f\indx{X}{X}$ is also a contraction, because it can be written as a
  composition of three contractions $f\indx{X}{X} =
  \justxmap{X}{\justxmap{A\oplus X}{\justxmap{B\oplus
        X}{X}{\pi_X}}{f}}{\iota_X}$. Then by
  \cref{lem:contraction_complementary}, $\id{X}-f\indx{X}{X}$ is an EP
  map. For the rest of this proof, we assume that all idempotents
  split; this is without loss of generality because we can pass to the
  dagger idempotent completion. Note that the dagger idempotent
  completion still has pseudoinverses {\cite{Cockett-Lemay}}. Because
  $\id{X}-f\indx{X}{X}$ is an EP map, by \cref{prop:ep-svd}, we can
  write
  \[\smallblocks
    \id{X}-f\indx{X}{X} = \;\;\; \begin{blockarray}{ccccccccc}
      &&&&&& \lab{X_1} & \lab{\oplus} & \lab{X_2}\hpad \\[0.25ex]
      \begin{block}{ccccc(cccc)}
        \lab{X_1}&&& &&& a && 0\vpad\hpad \\[0.25ex]
        \mathclap{\lab{\oplus}}\\
        \lab{X_2}&&& &&& 0 && 0\hpad \\
      \end{block}
    \end{blockarray}\;\;\;,
    \vspace{-0.25cm}
  \]
  where we have decomposed $X$ into a sum of two objects $X_1\oplus
  X_2$ and $\map{a}{X_1}{X_1}$ is invertible.
  Writing $f$ in matrix form, we now have
  \[
  \smallblocks
  f
  \qquad = \quad\begin{blockarray}{ccccccccccc}
    &&&&&& \lab{A} & \lab{\oplus} & \lab{X_1} & \lab{\oplus} & \lab{X_2}  \\[0.25ex]
  \begin{block}{ccccc(cccccc)}
    \lab{B}&&& &&& f\indx{A}{B} && f\indx{X_1}{B} && f\indx{X_2}{B}\vpad \\[0.25ex]
    \lab{\oplus}\\
    \lab{X_1}&&& &&& f\indx{A}{X_1} &&  \id{X_1} - a && 0 \\[0.25ex]        
    \lab{\oplus}\\                                                                                    
    \lab{X_2}&&& &&& f\indx{A}{X_2} && 0 && \,\id{X_2}\, \\
  \end{block}
  \end{blockarray}
  ~~~.
  \]
  Using \cref{cor:maxed}, we get $f\indx{A}{X_2}=0$ and
  $f\indx{X_2}{B}=0$. To show that the kernel-image trace of $f$ is
  defined, we must show that there exist $i$ and $k$ to complete the
  following diagram:
  \vspace{-1ex}
  \[
  \begin{tikzpicture}[yscale=0.9]
    \node (a) at (-1.5,1) {$A$};
    \node (b) at (.5,1) {$X$};
    \node (c) at (-.5,-1) {$X$};
    \node (d) at (1.5,-1) {$B$};
    \draw[->] (b) -- node[fill=white]{$\psmall{a & 0 \\ 0 & 0}$} (c);
    \draw[->] (a) -- node[left]{$\psmall{f\indx{A}{X_1}\\0}$} (c);
    \draw[->] (b) -- node[right]{$\psmall{f\indx{X_1}{B}&0}$} (d);
    \draw[->, dashed] (a) -- node[above]{$i$} (b);
    \draw[->, dashed] (c) -- node[below]{$k$} (d);
  \end{tikzpicture}
  \vspace{-2ex}
  \]
  But this can be achieved with
  $i=\pmat{\comp{f\indx{A}{X_1}}{a\inv}\\0}$ and
  $k=\pmat{\comp{a\inv}{f\indx{X_1}{B}}&0}$.
\end{proof}

In the next lemma, we do not assume pseudoinverses, so the
kernel-image trace of a given isometry may not exist. However, we show
that if it does exist, it is an isometry.

\begin{lemma}[Trace of isometry]
  \label{lem:kerim_isom}
  \label{lem:step2}
  In a dagger additive category, the
  kernel-image trace of an isometry, if it exists, is an isometry.
\end{lemma}

\begin{proof}
  Consider an arrow $\map{f}{A \oplus X}{B \oplus X}$ with components
  \[
  \smallblocks
  f = \;\;\;\begin{blockarray}{ccccccccc}
      &&&&&& \lab{A} & \lab{\oplus} & \lab{X} \\[0.25ex]
      \begin{block}{ccccc(cccc)}
        \lab{B}&&& &&& f\indx{A}{B} && f\indx{X}{B} \\[0.25ex]
        \lab{\oplus}\\
        \lab{X}&&& &&& f\indx{A}{X} && f\indx{X}{X} \\
      \end{block}
    \end{blockarray}
  \vspace{-0.25cm}
  \]
  Assume that $f$ is an isometry, so that we have
  \begin{align}
    \begin{split}\label{eqn:isom}
      \comp{f\indx{X}{B}}{f\indx{X}{B}^\dag} + \comp{f\indx{X}{X}}{f\indx{X}{X}^\dag} &= \id{X},\\
      \comp{f\indx{A}{B}}{f\indx{A}{B}^\dag} + \comp{f\indx{A}{X}}{f\indx{A}{X}^\dag} &= \id{A},\\
      \comp{f\indx{A}{B}}{f\indx{X}{B}^\dag} + \comp{f\indx{A}{X}}{f\indx{X}{X}^\dag} &= 0.
    \end{split}
  \end{align}
  Also assume that $\kitr{X}f$ exists, so in particular there exists
  $\map{i}{A}{X}$ satisfying
  \begin{equation}\label{eqn:im}
    f\indx{A}{X} = \comp{i}{(\id{X} - f\indx{X}{X})}.
  \end{equation}
  \par\pagebreak
  \noindent To show that $\kitr{X}f$ is an isometry, we calculate
  \vspace{-5ex}\par
  \[
  \renewcommand{\binarycomp}[2]{#2\!\circ\!#1}
  \scalebox{0.99}{\begin{minipage}{\textwidth}
  \begin{align*}
    \comp{\kitr{X}f}{(\kitr{X}f)^\dag}
    &= \comp{(f\indx{A}{B} + \comp{i}{f\indx{X}{B}})}{(f\indx{A}{B}^\dag + \comp{f\indx{X}{B}^\dag}{i^\dag})}\\
    &= \comp{f\indx{A}{B}}{f\indx{A}{B}^\dag}
      + \comp{i}{f\indx{X}{B}}{f\indx{A}{B}^\dag}
      + \comp{f\indx{A}{B}}{f\indx{X}{B}^\dag}{i^\dag}
      + \comp{i}{f\indx{X}{B}}{f\indx{X}{B}^\dag}{i^\dag}\\
    \text{(by \eqref{eqn:isom})}\quad
    &= \comp{f\indx{A}{B}}{f\indx{A}{B}^\dag}
      - \comp{i}{f\indx{X}{X}}{f\indx{A}{X}^\dag}
      - \comp{f\indx{A}{X}}{f\indx{X}{X}^\dag}{i^\dag}
      + \comp{i}{(\id{X} - \comp{f\indx{X}{X}}{f\indx{X}{X}^\dag})}{i^\dag}\\
    \text{(by \eqref{eqn:im})}\quad
    &= \comp{f\indx{A}{B}}{f\indx{A}{B}^\dag}
      - \comp{i}{f\indx{X}{X}}{(\id{X}-f\indx{X}{X}^\dag)}{i^\dag}
      - \comp{i}{(\id{X}-f\indx{X}{X})}{f\indx{X}{X}^\dag}{i^\dag} 
      + \comp{i}{(\id{X} - \comp{f\indx{X}{X}}{f\indx{X}{X}^\dag})}{i^\dag}\\
    &= \comp{f\indx{A}{B}}{f\indx{A}{B}^\dag}
      + \comp{i}{(-f\indx{X}{X}+\comp{f\indx{X}{X}}{f\indx{X}{X}^\dag})}{i^\dag}
      + \comp{i}{(-f\indx{X}{X}^\dag+\comp{f\indx{X}{X}}{f\indx{X}{X}^\dag})}{i^\dag} 
      + \comp{i}{(\id{X} - \comp{f\indx{X}{X}}{f\indx{X}{X}^\dag})}{i^\dag}\\
    &= \comp{f\indx{A}{B}}{f\indx{A}{B}^\dag}
      + \comp{i}{(\id{X}
      - f\indx{X}{X}^\dag
      - f\indx{X}{X}
      + \comp{f\indx{X}{X}}{f\indx{X}{X}^\dag})}{i^\dag}\\
    &= \comp{f\indx{A}{B}}{f\indx{A}{B}^\dag}
      + \comp{i}{(\id{X} - f\indx{X}{X})}{(\id{X} - f\indx{X}{X}^\dag)}{i^\dag}\\
    \text{(by \eqref{eqn:im})}\quad
    &= \comp{f\indx{A}{B}}{f\indx{A}{B}^\dag}
      + \comp{f\indx{A}{X}}{f\indx{A}{X}^\dag}\\
    \text{(by \eqref{eqn:isom})}\quad
    &= \id{A}.
  \end{align*}
\end{minipage}}
\]
\par
  \vspace{-6ex}
  \qedhere
\end{proof}

Note that the proof only used the $i$ of the kernel-image trace and
not the $k$. We use this fact to immediately obtain the following.

\pagebreak[2]

\begin{lemma}[Trace of contraction]\label{lem:kerim_contr}
  In a dagger additive category, the
  kernel-image trace of a contraction, if it exists, is a contraction.
\end{lemma}

\begin{proof}
  Suppose $\map{f}{A \oplus X}{B \oplus X}$ is a contraction. By the
  definition of contraction, there exists an object $B'$ and an
  arrow $\map{g}{A\oplus X}{B'}$ such that $\comp{f}{f^\dag} + \comp{g}{g^\dag}
  = \id{A\oplus X}$, or in other words, such that
  \[
  \smallblocks
  h = \;\;\;\begin{blockarray}{ccccccccc}
      &&&&&& \lab{A} & \lab{\oplus} & \lab{X} \\[0.25ex]
      \begin{block}{ccccc(cccc)}
        \lab{B'}&&& &&& g\indx{A}{B'} && g\indx{X}{B'} \\[0.25ex]
        \lab{\oplus}\\
        \lab{B}&&& &&& f\indx{A}{B} && f\indx{X}{B} \\
        \lab{\oplus}\\
        \lab{X}&&& &&& f\indx{A}{X} && f\indx{X}{X} \\
      \end{block}
    \end{blockarray}
    \vspace{-0.25cm}
  \]
  is an isometry.  Now assume that the kernel-image trace of $f$
  exists. While this does not necessarily imply that the kernel-image
  trace of $h$ exists, we nevertheless get the existence of $\map{i}{A}{X}$
  such that $f\indx{A}{X} = \comp{i}{(\id{X} - f\indx{X}{X})}$. As seen in the
  proof of \cref{lem:step2}, this is sufficient to show that
  \[
    \pmat{g\indx{A}{B'}\\f\indx{A}{B}} + \comp{i}{\pmat{g\indx{X}{B'}\\f\indx{X}{B}}} 
  =
  \pmat{g\indx{A}{B'} + \comp{i}{g\indx{X}{B'}} \\f\indx{A}{B} + \comp{i}{f\indx{X}{B}}} 
  \]
  is an isometry. Thus, the kernel-image trace of
  $f$, which is $f\indx{A}{B} + \comp{i}{f\indx{X}{B}}$, is a contraction, as
  claimed.
\end{proof}

We are now ready to prove our main theorem.

\begin{proof}[Proof of \cref{thm:main}]
  Lemmas~\ref{lem:kerim_total}, \ref{lem:kerim_isom}, and
  \ref{lem:kerim_contr} show that the kernel-image trace of the
  ambient category is total on isometries and contractions. Dually,
  the same holds for coisometries. Hence the trace is also total on
  their intersection, the unitaries. Moreover, in the cases of
  unitaries and contractions, the trace is a dagger trace by
  \cref{rem:dagger-trace}.
\end{proof}

\section*{Conclusion and future work}

We showed that in every pseudoinverse dagger additive category, each
of the subcategories of isometries, coisometries, unitary maps, and
contractions forms a totally traced category. This generalizes a
result by Bartha in the case of finite dimensional Hilbert spaces. One
of the main ingredients of this construction is the notion of
pseudoinverse, which was originally studied for matrices by Moore and
Penrose, but makes sense in any dagger category. Contractions can also
be defined in any dagger category (as compositions of isometries and
coisometries), but they only behave as expected if one assumes
additive structure and definiteness. The latter follows from the
existence of pseudoinverses.

One might ask whether Bartha's trace has a physical interpretation. We
do not know the answer, but some potential evidence to the contrary is
that the trace on contractions is not a continuous operation, and that
it does not exist in infinite dimensional spaces. See
\cref{app:physical} for more details.

In future work, it would be interesting to investigate whether the
assumptions under which contractions are traced could be further
reduced.
In fact, there are examples of dagger additive categories in which the
contractions are totally traced but not all pseudoinverses exist; see
\cref{cex:trace-without-pseudo}. On the other hand, it is not
sufficient to merely assume, say, the existence of dagger kernels; see
\cref{rem:no-infinite-trace}.

\bibliographystyle{eptcs}
\bibliography{references}

\newpage
\appendix

\section{The pseudotrace is not a trace}\label{app:pseudotrace}

In the proof of our main theorem, pseudoinverses
play a minor, but crucial role: they are only used to prove that the
kernel-image trace is total. In Bartha's original work
{\cite{Bartha_trace}}, pseudoinverses play a larger part, because he
uses them directly to define the trace on the category of finite
dimensional Hilbert spaces and isometries via the following formula:
\begin{equation}\label{eqn:pseudotrace}
  \mptr{X} f = f\indx{A}{B} + \comp{f\indx{A}{X}}{\mpi{(\id{X} - f\indx{X}{X})}}{f\indx{X}{B}}.
\end{equation}
In fact, the above formula is defined for all linear maps $f$ (not
necessarily isometries), and, as we show below, it agrees with the
kernel-image trace whenever the latter exists. However, Bartha's
operation is not a trace on the category of all linear maps, because
it fails to satisfy the trace axioms. We call it the
\emph{pseudotrace}.

\begin{definition}[Pseudotrace]\label{def:pseudotrace}
  In a dagger additive category, the
  \emph{pseudotrace} of $\map{f}{A \oplus X}{B \oplus X}$ is defined
  by {\eqref{eqn:pseudotrace}}, if the pseudoinverse of $\id{X} - f\indx{X}{X}$
  exists, and undefined otherwise. In particular, if the category has
  pseudoinverses, this is a totally defined operation.
\end{definition}
  
\begin{warning}
  In general, the pseudotrace is not a (partial or total) trace. In
  $\fdhilbc$, let $X=\C^2$ and $A=\C$, and consider $\map{f}{A \oplus
    X}{A \oplus X}$ and $\map{g}{X}{X}$ defined by
  \[
    f \; = \; \left(\begin{array}{c|cc} 0 & 1 & 0 \\\hline 1 & 1 & 0\\ 0 & 0 & 1\end{array}\right)
    \qquad \text{and} \qquad
    g \; = \; \pmat{1 & -1 \\ 0 & 1}.
  \]
  Then we find
  \[
  \mptr{X}(\comp{f}{(\id{A} \oplus g)}) = 0 \qquad \text{and} \qquad
  \mptr{X}(\comp{(\id{A} \oplus g)}{f}) = -1,
  \]
  violating dinaturality (one of the laws of traces)
  {\cite{Malherbe-Scott-Selinger}}.
\end{warning}

\begin{lemma}[Pseudotrace and kernel-image trace]\label{lem:kerim_is_mp}
  In a dagger additive category, whenever the
  pseudotrace and kernel-image trace are both defined, they coincide.
\end{lemma}

\begin{proof}
  Let $\map{f}{A \oplus X}{B \oplus X}$ be an arrow with both
  $\kitr{X}f$ and $\mptr{X}f$ defined. Taking $i$ and $k$ as in
  \cref{def:kernel-image}, we have
  \begin{align*}
    \kitr{X}f
    &= f\indx{A}{B} + \comp{i}{(\id{X} - f\indx{X}{X})}{k}\\
    &= f\indx{A}{B} + \comp{i}{(\id{X} - f\indx{X}{X})}{\mpi{(\id{X} - f\indx{X}{X})}}{(\id{X} - f\indx{X}{X})}{k}\\
    &= f\indx{A}{B} + \comp{f\indx{A}{X}}{\mpi{(\id{X} - f\indx{X}{X})}}{f\indx{X}{B}}\\
    &= \mptr{X}f.\qedhere
  \end{align*}
\end{proof}

\section{Non-physicality of the trace?}
\label{app:physical}

One may ask whether the trace operation on Hilbert spaces and
unitaries (or isometries, or contractions) has a physical
interpretation, e.g., whether there is some physical device that can
perform this operation when presented with an input unitary in the
form of a black box. One potential issue is that the trace is not a
continuous operation, i.e., an infinitesimal variation in the input
may cause a large variation in the output. Another potential issue is
that neither the pseudotrace (\cref{def:pseudotrace}) nor the
kernel-image trace is total on infinite dimensional Hilbert spaces,
even when restricted to contractions, isometries, coisometries, or
unitaries. The next two remarks make this precise.

\begin{remark}[Non-continuity of trace]
  The trace on finite dimensional Hilbert spaces with unitary maps is
  not a continuous operation. Take for example the
  $\theta$-parameterized family of rotations
  \[\pmat{
      \cos(\theta) & -\sin(\theta) \\
      \sin(\theta) & \cos(\theta)
    }.
  \]
  The trace along the second row and column is $1$ for $\theta = 0$
  but is $-1$ for $0 < \theta < 2\pi$. On the other hand, the trace is
  continuous on strict contractions, because in that case the
  pseudoinverse in {\eqref{eqn:pseudotrace}} is an actual inverse,
  which is a continuous operation.
\end{remark}

\begin{remark}[Nonexistence of trace in infinite dimensions]
  \label{rem:no-infinite-trace}
  Let $\map{f}{\ell^2}{\ell^2}$ be the contraction on the Hilbert
  space of square-summable sequences that multiplies the $n$th term of
  every sequence by $\frac{1}{n}$. Consider the unitary map
  $\justmap{\ell^2 \oplus \ell^2}{\ell^2 \oplus \ell^2}$ defined by
  \[
  \pmat{
    -(\id{\ell^2} - f) & \sqrt{\id{\ell^2} - (\id{\ell^2} - f)^2} \\[1ex]
    \sqrt{\id{\ell^2} - (\id{\ell^2} - f)^2} & \id{\ell^2} - f
  }.
  \]
  
  The pseudotrace along the second row and column does not
  exist, as $f$ is not pseudoinvertible (i.e., $f$ does not have a
  closed image; see \cref{exa:hilb-pseudo}). Indeed, if there were
  an induced invertible map from the coimage of $f$ (here the entire
  space) to the image of $f$, then its inverse would have to be
  unbounded.
  Neither does the kernel-image trace exist, as $\sqrt{\id{\ell^2} -
    (\id{\ell^2} - f)^2}$ does not factor through $f$. Indeed, if
  there were $k$ with $\sqrt{\id{\ell^2} - (\id{\ell^2} - f)^2} =
  \comp{f}{k}$, then $k$ would have to be unbounded.
\end{remark}

\section{More on pseudoinverses}\label{app:pseudoinverses}

Pseudoinverses arise inevitably in relation to dagger
idempotents. Recall that a morphism of idempotents $\map{f}{p}{q}$ is
an arrow $f$ such that $\comp{p}{f}{q} = f$.
As shown in \cite{Cockett-Lemay}, the pseudoinvertible arrows in any
dagger category $\Cc$ exactly correspond to isomorphisms of dagger
idempotents (that is, isomorphisms in the dagger idempotent completion
of $\Cc$):

\begin{proposition}[Pseudoinverses via dagger idempotents \cite{Cockett-Lemay}]
  \label{prop:mpi_dagsplit}
  In a dagger category, an arrow $\map{f}{A}{B}$ is pseudoinvertible
  if and only if there are dagger idempotents $\map{p}{A}{A}$ and
  $\map{q}{B}{B}$ such that $f$ is an isomorphism of dagger idempotents
  $\map{f}{p}{q}$.
  Furthermore, the inverse isomorphism of dagger
  idempotents $\justmap{q}{p}$ is given by $\mpi{f}$, and we have
  $p = \comp{f}{\mpi{f}}$ and $q = \comp{\mpi{f}}{f}$.
\end{proposition}

\begin{proof}
  To say $\map{f}{p}{q}$ is an isomorphism of dagger idempotents with
  inverse $\map{g}{q}{p}$ means $f = \comp{p}{f}{q}$ and
  $g = \comp{q}{g}{p}$ with $\comp{f}{g} = p$ and $\comp{g}{f} =
  q$. Since $p$ and $q$ are dagger idempotents, we have
  $(\comp{f}{g})^\dag = \comp{f}{g}$ and
  $(\comp{g}{f})^\dag = \comp{g}{f}$. Moreover $\comp{f}{g}{f} = f$
  and $\comp{g}{f}{g} = g$, so $g = \mpi{f}$.

  Conversely, assume $f$ is pseudoinvertible and let
  $p = \comp{f}{\mpi{f}}$ and $q = \comp{\mpi{f}}{f}$. We have that
  $\map{f}{p}{q}$ and $\map{\mpi{f}}{q}{p}$ are morphisms of dagger
  idempotents, because $f = \comp{f}{\mpi{f}}{f}{\mpi{f}}{f}$ and
  $\mpi{f} = \comp{\mpi{f}}{f}{\mpi{f}}{f}{\mpi{f}}$. The compositions
  $\comp{f}{\mpi{f}}$ and $\comp{\mpi{f}}{f}$ are respectively the
  identities on $p$ and $q$.
\end{proof}

Unlike inverses, pseudoinverses do not in general compose; see
\cref{cex:mpi_no_compose}. Still, we do have the following.

\begin{corollary}[Composition of pseudoinverses]\label{lem:mpi_compose}
  In a dagger category, if $\map{f}{A}{B}$ and $\map{g}{B}{C}$ are
  pseudoinvertible with $\comp{\mpi{f}}{f} = \comp{g}{\mpi{g}}$, then
  $\comp{f}{g}$ is pseudoinvertible with $\mpi{(\comp{f}{g})} =
  \comp{\mpi{g}}{\mpi{f}}$. Moreover, $\comp{\mpi{g}}{\mpi{f}}{f}{g}
  = \comp{\mpi{g}}{g}$ and $\comp{f}{g}{\mpi{g}}{\mpi{f}} =
  \comp{f}{\mpi{f}}$.
\end{corollary}
\begin{proof}
  This follows from \cref{prop:mpi_dagsplit}, as a composition of
  isomorphisms of idempotents is an isomorphism of idempotents.
\end{proof}

In a dagger additive category, we obtain four dagger idempotents of
interest (written in the matrix form of \cref{prop:svd}, assuming that
the dagger splittings exist):
\begin{align*}  
    \comp{\mpi{f}}{f}
    &= \smallblocks
    \begin{blockarray}{ccccccccc}
      &&&&&& \lab{\im f} & \lab{\;\oplus} & \lab{\;(\im f)^\perp\!\!\!\!} \\[0.25ex]
      \begin{block}{ccccc(cccc)}
        \lab{\im f}&&& &&& \justid && 0 \\[0.25ex]
        \mathclap{\lab{\oplus}}\\
        \lab{(\im f)^\perp}&&& &&& 0 && 0 \\
      \end{block}
    \end{blockarray}
      \;
    &\;
      \comp{f}{\mpi{f}}
      &= \smallblocks
      \begin{blockarray}{ccccccccc}
      &&&&&& \lab{(\ker f)^\perp} & \lab{\oplus} & \lab{\;\ker f\;} \\[0.25ex]
      \begin{block}{ccccc(cccc)}
        \lab{(\ker f)^\perp}&&& &&& \!\!\!\!\justid && 0 \\[0.25ex]
        \lab{\oplus}\\
        \lab{\ker f}&&& &&& \!\!\!\!0 && 0 \\
      \end{block}
    \end{blockarray}
    \\
    \justid - \comp{\mpi{f}}{f}
    &=\smallblocks
      \begin{blockarray}{ccccccccc}
        &&&&&& \lab{\im f} & \lab{\;\;\oplus} & \lab{\;(\im f)^\perp\!\!\!\!} \\[0.25ex]
        \begin{block}{ccccc(cccc)}
          \lab{\im f}&&& &&& 0 && 0 \\[0.25ex]
          \lab{\oplus}\\
          \lab{(\im f)^\perp}&&& &&& 0 && \justid \\
        \end{block}
      \end{blockarray}
      \;
    &\;
      \justid - \comp{f}{\mpi{f}}
    &= \smallblocks
      \begin{blockarray}{ccccccccc}
        &&&&&& \lab{(\ker f)^\perp} & \lab{\oplus} & \lab{\;\ker f\;} \\[0.25ex]
        \begin{block}{ccccc(cccc)}
          \lab{(\ker f)^\perp}&&& &&& \!\!\!\!0 && 0 \\[0.25ex]
          \lab{\oplus}\\
          \lab{\ker f}&&& &&& \!\!\!\!0 && \justid \\
        \end{block}
      \end{blockarray}
\end{align*}
\vspace{-0.6cm}

\noindent
They are, respectively, the projections onto the image, the coimage,
the cokernel, and the kernel of $f$.  The following propositions show
that these names are justified.

\begin{proposition}[Image via pseudoinverse]\label{prop:correct_im}
  Let $\map{f}{A}{B}$ be a pseudoinvertible arrow in a dagger category
  such that $\comp{\mpi{f}}{f}$ splits via the mono
  $\map{m}{X}{B}$. Then $m$ is the image of $f$ (i.e., the universal
  subobject through which $f$ factors).
\end{proposition}

\begin{proof}
  We have $m = \comp{m}{\mpi{f}}{f}$, so $m$ factors through every arrow
  that $f$ factors through.
\end{proof}

Note that $\mpi{f}$ and $f^\dag$ have the same image projection,
namely $\comp{f}{\mpi{f}} = \comp{f^\dagmpi}{f^\dag}$. This is also
the coimage projection of $f$ (and of $f^\dagmpi$).  Dually,
$\mpi{f}$ and $f^\dag$ have the same coimage projection, namely
$\comp{\mpi{f}}{f} = \comp{f^\dag}{f^\dagmpi}$, which is also the
image projection of $f$ (and of $f^\dagmpi$).

\begin{proposition}[Kernel via pseudoinverse]\label{prop:kernels}
  Let $\map{f}{A}{B}$ be a pseudoinvertible arrow in a dagger finite
  biproduct category, with $p$ and $\comp{f}{\mpi{f}}$ complementary
  dagger idempotents. Then $f$ has a (dagger) kernel (in the standard
  sense of \cite{Heunen-Karvonen}) if and only if $p$ (dagger)
  splits. The splitting is given by the inclusion map of the kernel.
\end{proposition}

\begin{proof}
  Observe that for all arrows $\map{m}{X}{A}$, we have $\comp{m}{f}=0$
  if and only if $m = \comp{m}{p}$. Indeed, if $\comp{m}{f}=0$, then
  $m = \comp{m}{(\comp{f}{\mpi{f}} + p)} = \comp{m}{p}$. Conversely,
  if $m = \comp{m}{p}$ then
  $\comp{m}{f} = \comp{m}{p}{f}{\mpi{f}}{f} = 0$. A universal such
  arrow $m$ is equivalently characterized as a kernel of $f$ or as an
  equalizer of $p$ and $\id{A}$. But an equalizer of an idempotent and
  an identity is the same as a mono splitting the idempotent, as
  desired.
  Moreover, a dagger kernel is by definition a kernel that is an
  isometry. It suffices to observe that every splitting of a dagger
  idempotent by an isometry $m$ is a dagger splitting:
  \[
  e = \comp{m}{m^\dag}{e} = \comp{m^\dag}{e^\dag}{m^\dag} = m^\dag.
  \qedhere
  \]
\end{proof}

\begin{lemma}[Pseudoinvertible monos split]\label{lem:mp_split_mono}
  In a dagger category, every pseudoinvertible mono (dually, epi) is
  split by its pseudoinverse.
\end{lemma}

\begin{proof}
  Suppose $\map{f}{A}{B}$ is a pseudoinvertible mono. We have $f =
  \comp{f}{\mpi{f}}{f}$. Thus by cancellation $\comp{f}{\mpi{f}} =
  \id{A}$.
\end{proof}

\begin{proposition}
  Every pseudoinverse dagger additive category in which
  all dagger idempotents split is an abelian category (in the standard
  sense of \cite{MacLane}).
\end{proposition}

\begin{proof}
  By \cref{prop:kernels}, all kernels exist. Moreover, every
  mono $m$ is normal as $m$ is a kernel of
  $1 - \comp{\mpi{m}}{m}$, using both \cref{prop:kernels} and
  \cref{lem:mp_split_mono}. Dually, all cokernels exist and every
  epi is normal.
\end{proof}

Although this paper is about pseudoinverse dagger additive categories,
so far we have not given many examples.
\begin{example}
  Let $\F$ be any dagger subfield of $\C$ (i.e., a subfield closed
  under conjugation). Then $\mat{\F}$ is a pseudoinverse dagger
  additive category. Indeed, it was observed in \cite{Pearl} that the
  pseudoinverse of a matrix $f$ over an arbitrary dagger field exists
  if and only if
  $\rank(f) = \rank(\comp{f}{f^\dag}) = \rank(\comp{f^\dag}{f})$,
  and the rank of a matrix does not change upon passing to a larger
  field.
\end{example}

\begin{lemma}\label{lem:rational}
  In a pseudoinverse dagger additive category, $\Q$ embeds into the
  endomorphism ring at each non-zero object.
\end{lemma}
\begin{proof}
  Let $n \in \N_{\geq 1}$, and let $\map{\delta}{A}{A^n}$ be the
  canonical $n$-ary diagonal map. By symmetry, we have $\mpi{\delta} =
  \pmat{d ~ \cdots ~ d}$ for some $\map{d}{A}{A}$. Since $\delta$ is
  mono, by \cref{lem:mp_split_mono}, we have
  $\comp{\delta}{\mpi{\delta}} = \id{A}$, hence $n \cdot d =
  \id{A}$. Thus $n\cdot \id{A}$ has a multiplicative inverse. As $\Q$
  is the universal ring in which every natural number has a
  multiplicative inverse, and $\Q$ has no proper quotients, the result
  follows.
\end{proof}

\begin{example}[Free pseudoinverse dagger additive categories]
  Since pseudoinverse dagger additive categories are essentially
  algebraic, we can consider freely generated structures. It follows
  from \cref{lem:rational} that the free pseudoinverse dagger additive
  category on an object is equivalent to $\mat{\Q}$. But more exotic
  examples exist, such as the free pseudoinverse dagger category on an
  object with an endomorphism; see \cref{cex:non-field}.
\end{example}

\section{Counterexamples}
\label{app:counterexamples}
This section consists of several counterexamples, which preclude
various strengthenings of results in the paper.  Our main theorem
gives a sufficient condition for the monoidal subcategory of
contractions in a dagger additive category to be traced: the existence
of pseudoinverses. However, this is not a necessary condition.
\begin{counterexample}[Trace without
  pseudoinverses]\label{cex:trace-without-pseudo}
  The kernel-image trace on the dagger additive category $\mat{\Z}$ of
  integer-valued matrices is totally defined on contractions. Indeed,
  since contractions are equivalently submatrices of unitaries, they
  are the matrices with entries in $\set{-1,0,1}$ with at most one nonzero
  entry per row and per column, and one may check that all
  kernel-image traces of such matrices exist.
  However, not all arrows (even those of the form $\justid - f$ with
  $f$ a contraction) are pseudoinvertible, e.g., the matrix
  $\pmat{2}$.
\end{counterexample}

Given the examples we have seen so far, it is reasonable to ask
whether all pseudoinverse dagger additive categories are sub dagger
additive categories of matrices over the complex numbers. However,
this is not the case.
\begin{counterexample}[Non complex matrix pseudoinverse dagger
  additive category]\label{cex:non-field}
  The free pseudoinverse dagger additive category on an object $*$ and
  an arrow $\map{f}{*}{*}$ does not embed into any dagger additive
  category of matrices over a field. Indeed, given an endomorphism $f$
  of a finite dimensional vector space, the image eventually
  stabilizes with repeated application, i.e. (assuming relevant
  pseudoinverses exist),
  $\comp{\mpi{(f^n)}}{f^n} = \comp{\mpi{(f^{n+1})}}{f^{n+1}}$ for some
  $n$. But such $n$ can be arbitrarily high, so in the free instance
  this cannot happen for any particular $n$.
\end{counterexample}

We saw in \cref{rem:isometry_component} that in a dagger additive
category, every isometry is a component of a unitary. If moreover all
dagger idempotents dagger split, then every isometry is a
\emph{column} of a unitary, using \cref{lem:proj_decompose}. We note
that this stronger statement does not hold without the assumption of
dagger splittings.
\begin{counterexample}[Non unitary-column isometry]
  Consider the full dagger finite biproduct subcategory of $\fdhilbc$
  of spaces with dimension not equal to $1$. The inclusion of
  a $2$-dimensional subspace into a $3$-dimensional space is then not a
  column of a unitary.
\end{counterexample}

We saw in \cref{cor:maxed} that for a contraction in a definite dagger
additive category, any row or column with a $1$ has all other entries
$0$. This does not hold without assuming definiteness.

\begin{counterexample}[Non maxed-out row]
  In $\mat{\F_2}$, the matrix $\pmat{1 ~ 1 ~ 1}$
  is a coisometry.
\end{counterexample}

Next, we give some more counterexamples relating to
pseudoinverses. Pseudoinverses do not compose: in general we do not
have $\mpi{(\comp{f}{g})} = \comp{\mpi{g}}{\mpi{f}}$. However, this
does hold when the image projection of $f$ and the coimage projection
of $g$ coincide, as in \cref{lem:mpi_compose}. We observe that it is
not sufficient to merely assume that the image projection of $f$
factors through the coimage projection of $g$.

\begin{counterexample}[Non composition of pseudoinverses]\label{cex:mpi_no_compose}
  Consider the dagger idempotent $p = \psmall{1 & 0 \\ 0 & 0}$ and the
  invertible matrix $a = \psmall{1 & 1 \\ 0 & 1}$ in $\fdhilbc$. We
  have $\comp{p}{a} = p$, and thus
  $\mpi{(\comp{p}{a})} = \mpi{p} = p$, whereas
  $\comp{\mpi{a}}{\mpi{p}} = \comp{a\inv}{p} = \psmall{1 & -1 \\ 0 &
    0}$.
\end{counterexample}

On the other hand, the next example shows that it is still possible to
have $\mpi{(\comp{f}{g})} = \comp{\mpi{g}}{\mpi{f}}$ in cases where
the image projection of $f$ and coimage projection of $g$ are
incompatible (do not even commute); in particular, the sufficient
condition given in \cref{lem:mpi_compose} is not necessary. (See
\cite{Cockett-Lemay} for a necessary and sufficient condition for
pseudoinverses to compose.)

\begin{counterexample}[Composition of pseudoinverses]
  In the dagger category of finite sets and relations (i.e.,
  boolean-valued matrices), the pseudoinverse of the idempotent
  $p = \psmall{1 & 0 \\ 1 & 0}$ is the idempotent
  $\mpi{p} = \psmall{1 & 1 \\ 0 & 0}$. Hence
  $\comp{\mpi{(pp)}} = \comp{\mpi{p}}{\mpi{p}}$. However, the image
  projection $\comp{\mpi{p}}{p} = \psmall{1 & 1 \\ 1 & 1}$ does not
  commute with the coimage projection
  $\comp{p}{\mpi{p}} = \psmall{1 & 0 \\ 0 & 0}$.
\end{counterexample}

We saw in \cref{lem:mp_split_mono} that every pseudoinvertible mono is
a split mono. However, the converse does not hold in general.
\begin{counterexample}[Non pseudoinvertible split mono]
  In $\mat{\Z}$, $m = \psmall{1 \\ 1}$ is a split mono. If its
  pseudoinverse existed, it would also be the pseudoinverse in
  $\mat{\Q}$. However, the pseudoinverse in $\mat{\Q}$ is
  $\pmat{\frac{1}{2}~\frac{1}{2}}$, which is not in $\mat{\Z}$.
\end{counterexample}

The following counterexamples show that when we do not assume the
existence of all pseudoinverses, the pseudotrace
(\cref{def:pseudotrace}) may be defined in cases where the
kernel-image trace is undefined or vice versa (even restricting to
unitaries).

\begin{counterexample}[Non pseudotrace kernel-image trace]
  In $\mat{\Z}$, the unitary $\psmall{1 & 0 \\ 0 & -1}$ does not have
  a pseudotrace along the second row and column, as $1 - (-1) = 2$ is
  not pseudoinvertible. It does have a kernel-image trace equal to
  $1 + 0(2)0 = 1$.
\end{counterexample}

\begin{counterexample}[Non kernel-image trace pseudotrace]
  In $\mat{\Z[x]/\langle x^2 \rangle}$, the unitary
  $\psmall{-1 & x \\ x & 1}$ has pseudotrace along the second row and
  column, as $1 - 1 = 0$ is pseudoinvertible. It does not have a
  kernel-image trace, as $x$ is not of the form $\comp{i}{0}$ for any
  $i$.
\end{counterexample}

Finally, we give several counterexamples that show certain results about
dagger additive categories do not hold in the absence of negatives. It
follows from \cref{prop:kernels} that any isometry $m$ in a dagger
additive category is the dagger kernel of $1 -
\comp{m^\dag}{m}$. However, isometries need not be kernels if we do
not assume the existence of negatives.

\begin{counterexample}[Non kernel isometry]
  In the dagger finite biproduct category of finite sets and relations
  (i.e., boolean-valued matrices), the isometry $\psmall{1 \\ 1}$ is
  not a kernel.
\end{counterexample}

We saw in \cref{lem:proj_decompose} that complementary split
idempotents are tantamount to direct sum decompositions. In an
additive category, idempotents $p$ and $q$ are complementary if and
only if $p + q = 1$, which does not hold in the absence of
negatives.
\begin{counterexample}[Non direct sum idempotent sum]
  Consider the (dagger) finite biproduct category of finite sets and
  relations. Letting $p = q = \id{\set{*}}$, we have $p + q =
  \id{\set{*}}$, but $pq\neq 0$. Thus, $p$ and $q$ are not
  complementary. Also, both $p$ and $q$ are split via $\set{*}$, but
  $\set{*}$ is not isomorphic to $\set{*} \oplus \set{*}$.
\end{counterexample}

Complementary split idempotents are split by (co)kernels of one
another, by an argument similar to the proof of
\cref{prop:kernels}. Thus each idempotent can be recovered from the
other as the (necessarily unique) idempotent split by its kernel and
cokernel. Hence, in the absence of negatives, it is reasonable to ask
whether idempotents that are split by (co)kernels of one another are
necessarily complementary. However, this is not the case.

\begin{counterexample}[Non complementary mutual (co)kernels]
  In the finite biproduct category of bounded semilattices
  (equivalently, modules over the booleans), consider the following
  semilattice, with evident ``projection-like'' idempotents onto
  the shown sublattices:
  \[
    \qquad
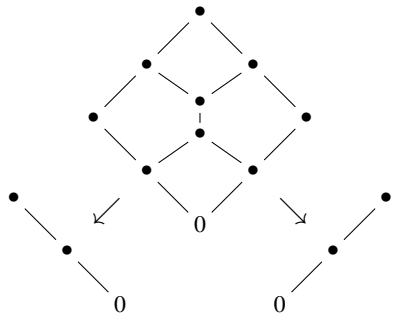
    \begin{tikzpicture}[every node/.style={inner sep=2,font=\footnotesize},rotate=45]
      \node (aa) at (0,0) {$0$};
      \node (ab) at (0,1) {$\bullet$};
      \node (ac) at (0,2) {$\bullet$};
      \node (ba) at (1,0) {$\bullet$};
      \node (bb1) at (.85,.85) {$\bullet$};
      \node (bb2) at (1.15,1.15) {$\bullet$};
      \node (bc) at (1,2) {$\bullet$};
      \node (ca) at (2,0) {$\bullet$};
      \node (cb) at (2,1) {$\bullet$};
      \node (cc) at (2,2) {$\bullet$};
      \draw (aa) -- (ab);
      \draw (ab) -- (ac);
      \draw (ba) -- (bb1);
      \draw (bb2) -- (bc);
      \draw (ca) -- (cb);
      \draw (cb) -- (cc);
      \draw (bb1) -- (bb2);
      \draw (aa) -- (ba);
      \draw (ba) -- (ca);
      \draw (ab) -- (bb1);
      \draw (bb2) -- (cb);
      \draw (ac) -- (bc);
      \draw (bc) -- (cc);
      \begin{scope}[shift={(-1.5,0)}]
        \node (arr) at (.75,1) {\normalsize$\swarrow$};
        \node (aa) at (0,0) {$0$};
        \node (ab) at (0,1) {$\bullet$};
        \node (ac) at (0,2) {$\bullet$};
        \draw (aa) -- (ab);
        \draw (ab) -- (ac);
      \end{scope}
      \begin{scope}[shift={(0,-1.5)}]
        \node (arr) at (1,.75) {\normalsize$\searrow$};
        \node (aa) at (0,0) {$0$};
        \node (ba) at (1,0) {$\bullet$};
        \node (ca) at (2,0) {$\bullet$};
        \draw (aa) -- (ba);
        \draw (ba) -- (ca);
      \end{scope}
    \end{tikzpicture}
  \]
  These idempotents are split by (co)kernels of one another, but they
  are not complementary.
\end{counterexample}

The kernel-image trace is a partial trace on any additive category. It
is reasonable to ask whether the same formula works in an arbitrary
finite biproduct category, simply leaving the trace undefined where
the relevant subtraction is not defined. However, this does not give a
partial trace in general.
\begin{counterexample}[Kernel-image non trace]
  Consider $\mat{\N[x, y] /\langle \mult{x}{y} \rangle}$. The matrix
  $\pmat{\mult{x}{y}} = 0$ has a negative, so the kernel-image trace
  formula (tracing out the entire input and output) $0 + 0(1 - 0)0$ is
  defined. On the other hand, the matrix $\pmat{\mult{y}{x}}$ does
  not have a negative, so the kernel-image trace formula is
  undefined. This violates the dinaturality law for partial traces
  {\cite{Malherbe-Scott-Selinger}}.
\end{counterexample}

In \cref{prop:contractions} we saw five equivalent characterizations
of contractions in a dagger additive category. However, these
conditions are not equivalent in a mere dagger finite biproduct
category (i.e., without assuming negatives). In fact, they are all
distinct, with implications between them as follows:
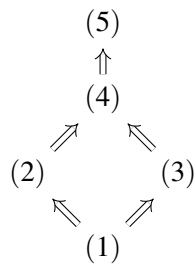
\[
  \begin{tikzpicture}
    \node (a) at (0,-1) {$\ref{item:c1}$};
    \node (b) at (-1,0) {$\ref{item:c2}$};
    \node (c) at (1,0) {$\ref{item:c3}$};
    \node (d) at (0,1) {$\ref{item:c4}$};
    \node (e) at (0,2) {$\ref{item:c5}$};
    \draw [imp] (a) -- (b);
    \draw [imp] (a) -- (c);
    \draw [imp] (b) -- (d);
    \draw [imp] (c) -- (d);
    \draw [imp] (d) -- (e);
  \end{tikzpicture}
\]

To distinguish them, it suffices to see that $\ref{item:c2}$ is distinct from
$\ref{item:c3}$ and that $\ref{item:c4}$ is distinct from $\ref{item:c5}$; it is then clear that the
self-dual definitions are distinct from the non-self-dual definitions.  To
say $\ref{item:c3}$ is distinct from $\ref{item:c4}$ means that contractions are not the
same as cocontractions.
\begin{counterexample}[Non cocontraction contraction]\label{cex:non-cocontraction}
  Consider the dagger finite biproduct category of finite sets and
  relations (i.e., boolean-valued matrices). The isometries are the
  matrices featuring at least one $1$ per column and at most one $1$
  per row. The matrix $\psmall{1 \\ 1}$ is an isometry, thus a
  contraction, but not a component of a coisometry, thus not a
  cocontraction.
\end{counterexample}

To say $\ref{item:c4}$ is distinct from $\ref{item:c5}$ means that not every coisometry
followed by an isometry is equal to an isometry followed by a
coisometry.
\begin{counterexample}[Non isometry then coisometry coisometry then
  isometry]\label{cex:non-iso-then-coiso}
  Consider the free dagger rig on an isometry $x$, i.e.,
  $\N[x, x^\dag]\langle \mult{x^\dag}{x} = 1 \rangle$. Its elements
  have explicit normal forms, as finite expressions
  $\sum_{i,j \geq 0} n_{i,j}\mult{x^j}{(x^\dag)^i}$. In the
  corresponding dagger finite biproduct category of matrices, the
  isometries are the matrices having entries in
  $\set{0, 1, x, x^2, \ldots}$ with one nonzero entry per column and
  at most one nonzero entry per row. Clearly the matrix
  $\pmat{\mult{x}{x^\dag}}$ cannot be expressed as an isometry
  followed by a coisometry.
\end{counterexample}

\end{document}